\documentclass[10pt]{article}
\usepackage{amssymb,amsmath,color}
\usepackage{amsthm}

\allowdisplaybreaks
\begin{document}

\def\pd#1#2{\frac{\partial#1}{\partial#2}}
\def\dfrac{\displaystyle\frac}
\let\oldsection\section
\renewcommand\section{\setcounter{equation}{0}\oldsection}
\renewcommand\thesection{\arabic{section}}
\renewcommand\theequation{\thesection.\arabic{equation}}

\def\Xint#1{\mathchoice
  {\XXint\displaystyle\textstyle{#1}}%
  {\XXint\textstyle\scriptstyle{#1}}%
  {\XXint\scriptstyle\scriptscriptstyle{#1}}%
  {\XXint\scriptscriptstyle\scriptscriptstyle{#1}}%
  \!\int}
\def\XXint#1#2#3{{\setbox0=\hbox{$#1{#2#3}{\int}$}
  \vcenter{\hbox{$#2#3$}}\kern-.5\wd0}}
\def\ddashint{\Xint=}
\def\dashint{\Xint-}

\newcommand{\cb}{\color{blue}}
\newcommand{\cred}{\color{red}}
\newcommand{\be}{\begin{equation}\label}
\newcommand{\ee}{\end{equation}}
\newcommand{\bea}{\begin{eqnarray}\label}
\newcommand{\beaa}{\begin{eqnarray}}
\newcommand{\eea}{\end{eqnarray}}
\newcommand{\nn}{\nonumber}
\newcommand{\intO}{\int_\Omega}
\newcommand{\Om}{\Omega}
\newcommand{\cd}{\cdot}
\newcommand{\pa}{\partial}
\newcommand{\ep}{\varepsilon}
\newtheorem{thm}{Theorem}
\newtheorem{lem}{Lemma}[section]
\newtheorem{proposition}{\indent Proposition}[section]
\newtheorem{dnt}{Definition}[section]
\newtheorem{remark}{Remark}
\newtheorem{cor}{Corollary}[section]
\allowdisplaybreaks

\title{Boundedness in a  quasilinear fully parabolic Keller-Segel
system of higher dimension with logistic source\thanks{{Supported by
the National Natural Science Foundation of China (11171048)}}}
\author{Cibing Yang \quad Xinru Cao \quad Zhaoxin Jiang \quad
Sining Zheng\thanks{Corresponding author.
E-mail: 1145250006@qq.com (C. Yang), caoxinru@gmail.com (X. Cao), jzxdlut@163.com (Z. Jiang), snzheng@dlut.edu.cn (S. Zheng)} \\
\footnotesize School of Mathematical Sciences, Dalian University of
Technology, Dalian 116024, P.R. China} \maketitle
\date{}
\maketitle {\bf Abstract.} This paper deals with the higher
dimension quasilinear parabolic-parabolic Keller-Segel system
involving a source term of logistic type $
u_t=\nabla\cdot(\phi(u)\nabla u)-\chi\nabla\cdot(u\nabla v)+g(u)$,
$\tau v_t=\Delta v-v+u$ in $\Omega\times (0,T)$, subject to
nonnegative initial data and homogeneous Neumann boundary condition,
where $\Omega$ is smooth and bounded domain in $\mathbb{R}^n$, $n\ge
2$,  $\phi$ and $g$ are smooth and positive functions satisfying
$ks^p\le\phi$ when $s\ge s_0>1$, $g(s) \le as - \mu s^2$ for $s>0$
with $g(0)\ge0$ and constants $a\ge 0$, $\tau,\chi,\mu>0$. It was
known that the model without the logistic source admits both bounded
and unbounded solutions, identified via the critical exponent
$\frac{2}{n}$. On the other hand, the model is just a critical case
with the balance of logistic damping and aggregation effects, for
which the property of solutions should be determined by the
coefficients involved. In the present paper it is proved that there
is $\theta_0>0$ such that the problem admits global bounded
classical solutions, regardless of the size of initial data and
diffusion whenever
$\frac{\chi}{\mu}<\theta_0$. This shows the substantial effect of the logistic source to the behavior of solutions.\\


{\bf Mathematics Subjection Classification}:  {92C17; 35K55; 35B35; 35B40.
\section{Introduction}

In this paper, we consider the higher dimension quasilinear
parabolic-parabolic Keller-Segel system with logistic source \be{a}
\left\{
\begin{array}{llc}
u_t=\nabla\cdot(\phi(u)\nabla u)-\nabla\cdot(\psi(u)\nabla v)+g(u), &(x,t)\in \Omega\times (0,T),\\[6pt]
\displaystyle
\tau v_t=\Delta v-v+u, &(x,t)\in\Omega\times (0,T),\\[6pt]
\displaystyle
\frac{\partial u}{\partial n}=\frac{\partial v}{\partial n}=0,&(x,t)\in \partial\Omega\times (0,T),\\[6pt]
u(x,0)=u_0(x),\quad  v(x,0)=v_0(x),
& x\in\Omega,
\end{array}
\right. \ee where $\Omega\subset\mathbb{R}^n$ ($n\ge 2$) is
 a bounded domain
with smooth boundary, $\tau>0$. Functions $\phi,\psi\in C^2([0,\infty))$
satisfy
\begin{align}\label{1.2}
   & \psi(s)=\chi s\\\label{1.3}
   & \phi(s)>0, ~~s\ge0,~~k s^p\le \phi(s),~~ s\ge s_0,
\end{align}
 with $\chi>0$, $k>0$, $p\in\mathbb{R}$, $s_0>1$, and $g \in C^{\infty}([0,\infty))$ fullfills
  \be{g}
    g(s)\le as-\mu s^2,~~ s>0
  \ee
with $g(0){\ge}0$, and constants $a\ge 0$, $\mu>0$. Here, $u$ and
$v$ represent the density of cells and the concentration of chemical
signal respectively. The classical Keller-Segel system can be
obtained by setting $\phi\equiv\psi\equiv 1$ and $g\equiv 0$ in
(\ref{a}) , which models the mechanism of chemotaxis, and has been
extensively studied since 1970, we refer to
\cite{H,jaeger_luckhaus,W4,W5} and the reference therein.

Eq.\,(\ref{a}) with $g(u)\equiv 0$ is a type of refined models
pursued by Hillen and Painter \cite{HP}, with the bacterial cells
having a positive size, the so-called volume-filling effect. Beyond
this, more general functions $\phi$ and $\psi$ are involved to
denote the diffusivity and chemotatic sensitivity, respectively
\cite{CW,HW,W3}. When $\phi(s)\sim s^p$ and $\psi(s)\sim s^q$ for
large $s$, a critical exponent $\frac{2}{n}$ on the interplay of
$\phi$ and $\psi$ has been found to identify boundedness and
unboundedness. Namely, if $q-p<\frac{2}{n}$, then all solutions are
global and uniformly bounded  \cite{WD,TyW}; however, if
$q-p>\frac{2}{n}$, unbounded solutions do exist \cite{W3}, even
finite-time blow-up may occur under some additional conditions $n\ge
3$ and $q\ge1$ \cite{CS,WD}.

 Apart from the aforementioned system, a source of logistic type is included in (\ref{a}) to describe
 the spontaneous growth of cells. The effect of preventing ultimate growth has been widely studied \cite{OY,OTYM,W1,TW}.
In the related classical semilinear chemotaxis systems, that is when
$\phi(u)\equiv 1$ and $\psi(u)=\chi u$ with $\chi>0$, such
proliferation mechanisms are known to prevent chemotactic collapse:
In \cite{TW}, for instance, it was proved that if $\tau=0$ and
$\mu>\frac{(n-2)_+}{n}\chi$, solutions of the parabolic-elliptic
system are global and remain bounded. The same conclusion is true
for the fully parabolic system with $\tau>0$ if either $n\le2$,
$\mu>0$ \cite{OY}, or $n\ge 3$ and $\mu>\mu_0$ with some constant
$\mu_0(\chi)>0$ \cite{W1}. This is in sharp contrast to the
possibility of blow-up which is known to occur in such systems when
$g\equiv 0$ and $n\ge 2$
\cite{herrero_velazquez,jaeger_luckhaus,nagai2001,W5}.

In this context, we intend to study (\ref{a}) with $\tau>0$ under
the conditions (\ref{1.2})--(\ref{g}). It is our purpose to
investigate the interaction among the triple of nonlinear diffusion,
aggregation and the logistic absorption. By taking $u^{\gamma-1}$ as
the test function to the first equation, and then substituting the
second equation, the standard $L^\gamma$ estimate argument yields
\begin{align}\nn
\frac{1}{\gamma}\frac{d}{dt}\intO u^\gamma\le -(\gamma-1)\intO u^{\gamma+p-2}|\nabla u|^2
+\frac{\gamma-1}{\gamma+q-1}\intO u^{\gamma+q}-\frac{\gamma-1}{\gamma+q-1}\intO u^{\gamma+q-1}\Delta v\\\label{Lp}
~~~~~~~~~~~~~~~~~~~~~~~~+a\intO u^{\gamma}-\mu\intO u^{\gamma+1}.
\end{align}
Comparing the terms $\displaystyle\frac{\gamma-1}{\gamma+q-1}\intO
u^{\gamma+q}$ and $\displaystyle-\mu\intO u^{\gamma+1}$, it is easy
to find that $q=1$ is critical. It has been proved that when $q<1$,
the logistic dampening rules out the occurrence of blow-up
regardless of diffusion \cite{C1}. And when $q>1$, the strong
diffusion  with $q-p<\frac{2}{n}$ ensures global boundedness by
\cite{TyW}, without the help of the logistic damping. The critical
case $q=1$ is more involved: from (\ref{Lp}) we may expect that
under the balance of logistic damping and aggregation effects, the
coefficients would determine weather the solution is bounded. In
\cite{C2}, it has been proved that when $q=1$ with $\tau=0$, the
solutions are bounded if $\mu>(1-\frac{2}{n(1-p)}_+)\chi$ for the
parabolic-elliptic case. This makes an agreement with the above
expectation. The main result of the present paper is the following
theorem for the fully parabolic Keller-Segel system.

\begin{thm}\label{th}

Suppose that $\Omega \subset \mathbb{R}^n$ ($n\ge 2$) is a bounded
domain with smooth boundary, and  $\chi,\mu,\tau>0$. Assume that
$\psi(u)=\chi u$, $\phi$ and $g$ satisfy (\ref{1.3})-(\ref{g}). Then
there is $\theta_0>0$ such that if $\frac{\chi}{\mu}<\theta_0$, for
any nonnegative $u_0\in C^0(\bar{\Omega})$ and $v_0\in
W^{1,r}(\Omega)$ with $r>n$, Eq.\,(\ref{a}) uniquely admits a
classical solution $(u,v)$ such that $u\in
C^0(\bar{\Omega}\times[0,\infty))\cap
C^{2,1}(\bar{\Omega}\times(0,\infty))$ and $v\in
C^0(\bar{\Omega}\times[0,\infty))\cap
C^{2,1}(\bar{\Omega}\times(0,\infty))\cap L_{\rm loc}^\infty
([0,T_{\max});W^{1,r}(\Omega))$. Moreover, $(u,v)$ is bounded in
$\Omega \times (0,\infty)$.
\end{thm}

\begin{remark}
{\rm We underline that the above result is independent of the value
of $p$ in (\ref{1.3}), and thus extends the analogue result for the
semilinear case \cite{W1}. Moreover, due to the technique used here,
the convexity of $\Om$ (required in \cite{W1}) is unnecessary in our
theorem.}
\end{remark}

  Unlike using the trace embedding technique to estimate the boundary integral in \cite{I_nonconvex,J_nonconvex},
our approach strongly relies on the {\em Maximal Sobolev
Regularity}.

The paper is arranged as follows. In section 2, we deal with the
local existence and the extensibility of classical solution to
(\ref{a}) as well as a variation of Maximal Sobolev Regularity.
Section 3 will be devoted to prove Theorem \ref{th}.

\section{ Preliminaries}

The local solvability to (\ref{a}) for sufficiently smooth initial
data can be addressed by methods involving standard parabolic
regularity theory in a suitable fixed point framework. In fact, one
can thereby also derive a sufficient condition for extensibility of
a given local-in-time solution. Details of the proof can be founded
in \cite{C1}.

\begin{lem}\label{lem2}
Suppose $\Omega\subset \mathbb{R}^n$ ($n\ge3$) is a bounded domain
with smooth boundary, $\phi$ and $\psi$ satisfy
(\ref{1.2})-(\ref{1.3}), $g$ fulfills (\ref{g}), $u_0\in
C^0(\bar{\Omega})$ and $v_0\in W^{1,r}({\Omega})$ (with some $r>n$)
both are nonnegative. Then there exists  $(u,v)\in
(C^0(\bar{\Omega}\times[0,T_{\max}))\cap
C^{2,1}(\bar{\Omega}\times(0,T_{\max})))^2$ with
$T_{\max}\in(0,\infty)$ classically solving (1.1) in
$\Omega\times(0,T_{\max})$. Moreover, if $T_{\max}<\infty$, then
\be{Tmax}\mathop{\limsup}\limits_{{t\nearrow T_{\max}}}
\|u(\cdot,t)\|_{L^{\infty}(\Omega)}=\infty. \ee
\end{lem}
Note that in (\ref{Lp}), the term $\displaystyle-\frac{\gamma-1}{\gamma+q-1}\intO u^{\gamma+q-1}\Delta v$ is unsigned.
 We thus intend to estimate its absolute value adequately.
For this purpose, we will make use of the following property
referred to as a variation of {Maximal Sobolev Regularity}, which
will play an important role in the proof of our main result. The following Lemma is not
the original version of a corresponding statement in \cite[Theorem 3.1]{MJ}, but by means of a simple transformation
by including an exponential weight function as in \cite{C}. 
\begin{lem}\label{lem2.2}
Let $r\in(1,\infty),\tau>0$. Consider the following evolution
equation
\begin{eqnarray}\label{6} \left\{
\begin{array}{llll}\displaystyle \tau v_t=\Delta v
-v+u,&(x,t)\in\Om\times(0,T),\\[4pt]
\displaystyle\frac{\partial
v}{\partial \nu}=0,&(x,t)\in\partial\Om\times(0,T),\\[4pt]
\displaystyle v(x,0)= v_0(x),&x\in\Om.\\[4pt]
\end{array}\right. \end{eqnarray}
For each $v_0\in W^{2,r}(\Omega)$ $(r>n)$ with $\frac{\partial
v_0}{\partial \nu}=0$ on $\partial\Omega$ and any $u\in
L^r((0,T);L^r(\Omega))$, there exists a unique solution
\begin{eqnarray*}
v\in W^{1,r}\big((0,T);L^r(\Omega)\big)\cap L^r\big((0,T);W^{2,r}(\Omega)\big).
\end{eqnarray*}
Moreover, there exists $C_r>0$, such that if $s_0\in[0,T)$,
$v(\cdot,s_0)\in W^{2,r}(\Omega)(r>n)$ with $\frac{\partial
v(\cdot,s_0)}{\partial n}=0$, then \beaa\label{2.2.1}
\int_{s_0}^T\intO {\rm e}^{\frac r{\tau}s}|\Delta v|^r\le
C_r\int_{s_0}^T\intO {\rm e}^{\frac r{\tau}s}u^r + C_r\tau {\rm
e}^{\frac r{\tau}s_0}\left(\|v(\cdot,s_0)\|_{L^r(\Omega)}^r+ \|\Delta
v(\cdot,s_0)\|_{L^r(\Omega)}^r\right). \eea
\end{lem}

\begin{proof}
Let $w(x,s)={\rm e}^sv(x,\tau s)$. We derive that $w$ satisfies
\begin{equation}\left\{
\begin{array}{llc}
w_s(x,s)=\Delta w(x,s)+{\rm e}^su(x,\tau s),&(x,s)\in\Om\times(0,T),\\
\frac{\pa w}{\pa\nu}=0,&(x,s)\in\partial\Om\times(0,T),\\
w(x,0)=v_0(x),&x\in\Om.
\end{array}
\right.
\end{equation}
Applying the Maximal Sobolev Regularity (\cite[Theorem 3.1]{MJ}) to
$w$, we obtain that \bea{est_MJ1} \int_0^T\intO|\Delta w(x,s)|^r \le
C_r\int_0^T\intO|{\rm e}^su(x,\tau s)|^r +
C_r\|v_0\|_{L^r(\Omega)}^r+C_r\|\Delta  v_0\|_{L^r(\Omega)}^r. \eea
Substituting $v$ into the above inequality and changing the
variables imply \beaa\nn \int_0^T\intO {\rm e}^{\frac
r{\tau}s}|\Delta v|^r\le C_r\int_0^T\intO {\rm e}^{\frac
r{\tau}s}u^r + C_r\tau\|v_0\|_{L^r(\Omega)}^r+C_r\tau \|\Delta
v_0\|_{L^r(\Omega)}^r. \eea Consequently, for any $s_0>0$, replacing $v(t)$ by
$v(t+s_0)$, we prove \eqref{2.2.1}.
\end{proof}

\section{Proof of Theorem \ref{th}}

In this section, we are going to prove our main result.
 Since the regularity obtained in (\ref{est_MJ1}) requires that the initial data satisfy homogeneous Neumann boundary conditions,
 we will perform a small time shift and thus use any positive time as the ``initial time" to guarantee
that the respective boundary condition is satisfied naturally.

 Specifically, given any $s_0\in(0,T_{\max})$ such that $s_0\le 1$, from the regularity principle asserted by Lemma \ref{lem2} we know that
$(u(\cdot,s_0),v(\cdot,s_0))\in C^2(\bar{\Omega})$ with $\frac{\partial v(\cdot,s_0)}{\partial \nu}=0$ on $\partial\Omega$, so that in particular we can pick $M>0$ such that
\bea{K}\mathop{\sup}\limits_{0\le\theta\le s_0}\|u(\cdot,\theta)\|_{L^{\infty}(\Omega)}\le M,
 \mathop{\sup}\limits_{0\le\theta\le s_0}\|v(\cdot,\theta)\|_{L^{\infty}(\Omega)}\le M,\text{ and }\|\Delta v(\cdot,s_0)\|_{L^{\infty}(\Omega)}\le M.\eea

Now we proceed to derive an a priori estimate which will constitute
the main part of the work.

\begin{lem}\label{lem33}
Suppose $\Om\subset \mathbb{R}^n$, $n\ge 3$, is a bounded domain
with smooth boundary, $\tau>0$ and $\chi\in \mathbb{R}$. For any
$\gamma>1$, $\eta>0$, there exist ${\mu}_{\gamma,\eta}>0$ and
$C=C(\gamma,|\Omega|,\mu,\chi,\eta,u_0,v_0)>0$ such that if
$\mu>\mu_{\gamma,\eta}$, then
$$\|u(\cdot,t)\|_{L^\gamma(\Om)}\le C$$
for all $t\in(s_0,T_{\max})$.
\end{lem}
\begin{proof}  We fix $s_0\in (0,T)$ such that $s_0\le 1$.
For arbitrary $\gamma>1$, take $u^{\gamma-1}$ as a test function for
the first equation in (\ref{a}) and integrate by part to obtain
\begin{align}\nn
\frac{1}{\gamma}\frac{d}{dt}\intO u^\gamma&=-(\gamma-1)\intO u^{\gamma-2}\phi(u)|\nabla u|^2+\chi(\gamma-1)\intO u^{\gamma-1}\nabla u\cdot\nabla v+a\intO u^\gamma-\mu\intO u^{\gamma+1}\\\nn
&\le \chi\frac{\gamma-1}{\gamma}\intO\nabla u^\gamma\cdot\nabla v+a\intO u^\gamma-\mu\intO u^{\gamma+1}\\\nn
&=-\chi\frac{\gamma-1}{\gamma}\intO u^\gamma\Delta v+a\intO u^\gamma-\mu\intO u^{\gamma+1}\\\label{3.1.1}
&= -\frac{\gamma+1}{\tau\gamma}\intO u^\gamma-\chi\frac{\gamma-1}{\gamma}\intO u^\gamma\Delta v+\bigg(a+\frac{\gamma+1}{\tau\gamma}\bigg)\intO u^\gamma-\mu\intO u^{\gamma+1}
\end{align}
for all $t\in(s_0,T_{\max})$. Here by Young's inequality, for any $\ep>0$, there exists $c_1>0$ such that
\begin{align}\label{3.1.2}
\bigg(a+\frac{\gamma+1}{\tau\gamma}\bigg)\intO u^\gamma\le \varepsilon\intO u^{\gamma+1}+c_1(a,\varepsilon,\gamma)|\Omega|,
\end{align}
where $c_1(a,\varepsilon,\gamma)=\frac{1}{\gamma}(1+\frac{1}{\gamma})^{-(\gamma+1)}\varepsilon^{-\gamma}
(a+\frac{\gamma+1}{\tau\gamma})^{\gamma+1}$. Young's inequality also implies that
\begin{align}\nn
-\chi\frac{\gamma-1}{\gamma}\intO u^\gamma \Delta v &\le \chi\intO
u^\gamma|\Delta v|\\\label{3.1.3} &\le \eta\intO
u^{\gamma+1}+c_2\eta^{-\gamma}\chi^{\gamma+1}\intO |\Delta
v|^{\gamma+1}
\end{align}
with
$c_2=\mathop{\sup}\limits_{\gamma>1}\frac{1}{\gamma}(1+\frac{1}{\gamma})^{-(\gamma+1)}<\infty$.
By  substituting (\ref{3.1.2}) and (\ref{3.1.3}) into (\ref{3.1.1}),
we  find that
\begin{align}\nn
\frac{d}{dt}\left(\frac{1}{\gamma}\intO u^\gamma\right)&\le -\frac{\gamma+1}{\tau}\left(\frac 1{\gamma}\intO u^\gamma\right)-(\mu-\ep-\eta)\intO u^{\gamma+1}+c_2\eta^{-\gamma}\chi^{\gamma+1}\intO|\Delta v|^{\gamma+1}\\\label{}
&~~~~~~+c_1(a,\varepsilon,\gamma)|\Om|
\end{align}
holds for all $t\in(s_0,T_{\max})$.
Applying the variation-of-constants formula to the above inequality shows that
\begin{align}\nn
\frac{1}{\gamma}\intO u^\gamma(\cdot,t)&\le {\rm
e}^{-\frac{\gamma+1}{\tau}(t-s_0)}\frac{1}{\gamma}\intO
u^\gamma(\cdot,s_0)-(\mu-\ep-\eta)\int_{s_0}^t{\rm
e}^{-\frac{\gamma+1}{\tau}(t-s)}\intO u^{\gamma+1}\\\nn
&~~~~~~+c_2\eta^{-\gamma}\chi^{\gamma+1}\int_{s_0}^t{\rm
e}^{-\frac{\gamma+1}{\tau}(t-s)}\intO |\Delta
v|^{\gamma+1}+c_1|\Om|\int_{s_0}^t {\rm
e}^{-\frac{\gamma+1}{\tau}(t-s)}\\\nn &\le -(\mu-\ep-\eta){\rm
e}^{-\frac{\gamma+1}{\tau}t}\int_{s_0}^t\intO {\rm
e}^{\frac{\gamma+1}{\tau}s}u^{\gamma+1}\\\label{3.1.4}
&~~~~~~+c_2\eta^{-\gamma}\chi^{\gamma+1}{\rm
e}^{-\frac{\gamma+1}{\tau}t}\int_{s_0}^t\intO {\rm
e}^{\frac{\gamma+1}{\tau}s} |\Delta v|^{\gamma+1}+c_3(a,\ep,\gamma,|\Om|,s_0)
\end{align}
for all $t\in(s_0,T_{\max})$, where
$$c_3(a,\varepsilon,\gamma,|\Omega|)=c_1|\Om|\int_{s_0}^t {\rm
e}^{-\frac{\gamma+1}{\tau}(t-s)}+\frac{1}{\gamma}\intO
u^\gamma(\cdot,s_0)$$ is independent of $t$. Next, we apply Lemma
\ref{lem2.2} to see that there is $C_{\gamma}>0$ such that
\begin{align}\nn
&~~c_2\eta^{-\gamma}\chi^{\gamma+1}{\rm
e}^{-\frac{\gamma+1}{\tau}t}\int^t_{s_0}\int_\Omega {\rm
e}^{\frac{\gamma+1}{\tau}s}|\Delta v|^{\gamma+1}\\\label{3.1.5}
&\leq c_2\eta^{-\gamma}\chi^{\gamma+1}{\rm
e}^{-\frac{\gamma+1}{\tau}t}\left(C_\gamma\int^t_{s_0}\int_\Omega
{\rm e}^{\frac{\gamma+1}{\tau} s}u^{\gamma+1}+C_\gamma\tau {\rm
e}^{\frac{\gamma+1}{\tau}s_0}\|v(\cdot,s_0)\|_{W^{2,\gamma+1}(\Om)}^{\gamma+1}\right).
\end{align}
Inserting \eqref{3.1.5} into \eqref{3.1.4} with some rearrangement, we finally arrive at
\begin{align}\nn
\frac{1}{\gamma}\intO u^{\gamma}(\cdot,t)&\le -\left(\mu-\ep-\eta-c_2C_{\gamma+1}p\eta^{-\gamma}\chi^{\gamma+1}\right){\rm e}^{-\frac{\gamma+1}{\tau}t}\int_s^t\intO {\rm e}^{\frac{\gamma+1}{\tau}s}u^{\gamma+1}+c_3\\
&~~~~~~~~~~+c_2C_{\gamma+1}\tau\eta^{-\gamma}\chi^{\gamma+1}{\rm
e}^{-\frac{\gamma+1}{\tau}(t-s_0)}
\|v(s_0)\|_{W^{2,\gamma+1}(\Om)}^{\gamma+1}
\end{align}
for all $t\in(s_0,T_{\max})$. Let ${\mu}_{\gamma,\eta}=\eta+c_2C_{\gamma+1}\eta^{-\gamma}\chi^{\gamma+1}$, we can choose $\ep\in(0,\mu-\mu_{\gamma,\eta})$ such that $$\mu-\ep-\eta-c_2C_\gamma\eta^{-\gamma}\chi^{\gamma+1}\ge0.$$
It is entailed that
\begin{align}\label{3.1.6}
\frac{1}{\gamma}\intO u^{\gamma}(\cdot,t)&\le c_4
\end{align}
for all $t\in(s_0,T_{\max})$ with
$c_4=c_3+c_2C_\gamma\tau\eta^{-\gamma}\chi^{\gamma+1}
\|v(s_0)\|_{W^{2,\gamma+1}(\Om)}^{\gamma+1}$. This completes the
proof by the above inequality together with \eqref{K}.
\end{proof}

Next, we invoke the well established Moser iteration to get boundedness of $(u,v)$.

\begin{proof} [Proof of Theorem \ref{th}]
By Morse's iteration (Lemma A.1 in \cite{TyW}), we claim that there
is $\gamma_0(n,p)>n>0$, determined via (A.8)--(A.10) in Lemma A.1 of
\cite{TyW}, such that if \bea{ga0}
\|u(\cdot,t)\|_{L^\gamma(\Om)}<\infty \eea for all
$\gamma\ge\gamma_0$ and all $t\in(s_0,T_{\max})$,  then there exists
$C_1>0$ such that \bea{inf} \|u(\cdot,t)\|_{L^\infty(\Om)}\le C_1
\eea for all $t\in(s_0,T_{\max})$. Actually, (\ref{ga0}) implies
that $\nabla v$ is bounded. It is easy to check that all assumptions
of Lemma A.1 are fulfilled.

 Let $\theta_0$ satisfy
$$\mathop{\inf}\limits_{\eta>0}\mu_{\eta,\gamma_0}=\mathop{\inf}\limits_{\eta>0}(\eta+c_2C_{\gamma_0+1}
\eta^{-\gamma_0}\chi^{\gamma_0+1})=\frac1{\theta_0}\chi.$$ We see
that $\frac{\chi}{\mu}<\theta_0$ implies $\mu>\mu_{\eta,\gamma_0}$
with some $\eta>0$. We know By Lemma \ref{lem33} that (\ref{ga0})
holds, and hence (\ref{inf}) is true. Combining with (\ref{K}), we
get that $u$ is bounded in $(0,T_{\max})$. The boundedness of $v$
can be obtained by the standard parabolic regularity. Finally, Lemma
\ref{lem2} yields that $(u,v)$ is global by contradiction.
\end{proof}

{\small
\end{document}